\RequirePackage{fix-cm}
\let\origvec\vec
\documentclass[smallextended]{svjour3}       

\let\vec\origvec
\smartqed  
\usepackage [latin1]{inputenc}
\usepackage{amsmath}
\usepackage{mathrsfs}
\usepackage{amssymb}
\usepackage{graphicx}
\usepackage[justification=centering]{caption}
\usepackage{lineno}
\usepackage{color}
 \numberwithin{equation}{section}

\journalname{myjournal}
\begin{document}
\title{The existence of stratified linearly steady two-mode water waves with stagnation points}


\author{Jun Wang    \and
         Fei Xu\and Yong Zhang
}


\institute{Jun Wang \at
         School of Mathematical Sciences, Jiangsu University, Zhenjiang 212013, People's Republic of China \\
             \email{wangmath2011@126.com}          
              \and
            Fei Xu \at
             School of Mathematical Sciences, Jiangsu University, Zhenjiang 212013, People's Republic of China \\
             \email{xufeiujs@126.com}
        \and
              Yong Zhang (Corresponding author) \at
              School of Mathematical Sciences, Jiangsu University, Zhenjiang 212013, People's Republic of China \\
              \email{18842629891@163.com}
}

\date{Received: date / Accepted: date}

\maketitle

\begin{abstract}

This paper focuses on the analysis of stratified steady periodic water waves that contain stagnation points. The initial step involves transforming the free-boundary problem into a quasilinear pseudodifferential equation through a conformal mapping technique, resulting in a periodic function of a single variable. By utilizing the theorems developed by Crandall and Rabinowitz, we establish the existence and formal stability of small-amplitude steady periodic
capillary-gravity water waves in the presence of stratified linear flows. Notably, the stability of bifurcation solution curves is strongly influenced by the stratified nature of the system. Additionally, as the Bernoulli's function $\beta$ approaches critical values, we observe that the linearized problem exhibits a two-dimensional kernel. Consequently, we apply a bifurcation theorem due to Kielh\"{o}fer that incorporates multiple-dimensional kernels and parameters, which enables us to establish the existence of two-mode water waves. As far as we know, the two-mode water waves in stratified flow are first constructed by us. Finally, we demonstrate the presence of internal stagnation points within these waves.
\end{abstract}



\section{\bf Introduction}

This paper is devoted to the investigation of two-dimensional
stratified steady periodic waves that propagate over a flat bed,
taking into account the effects of gravity and surface tension as
restoring forces. These waves represent solutions to the
incompressible Euler equations, considering the presence of a free
boundary. Notably, we consider waves with internal stagnation points
and overturning profiles, allowing for a more comprehensive
analysis. To formulate the problem, we consider an unbounded domain
$\Omega$ in the $(X, Y)$ plane, where the impermeable flat bottom is
given by
$$
\mathcal{B}:=\{ (X,0): X\in \mathbb{R} \}
$$
and the unknown surface is parameterized by
$$
\mathcal{S}:=\{(u(s),v(s)): s\in \mathbb{R}\},
$$
where the map $s\mapsto (u(s)-s,v(s))$ is periodic of period $2\pi$, $v(s)\geq 0$ and $u'^{2}(s)+v'^{2}(s)\neq0$.

Since the incompressibility implies that the velocity field of fluid is divergence free, which allows us to introduce the pseudo-stream function $\psi(X,Y)$ with satisfying
\begin{eqnarray}\label{eq1.1}
\psi_{X}=-\sqrt {\rho}w_2, ~~~~~~~\psi_{Y}=\sqrt {\rho}w_1,
\end{eqnarray}
where $(w_1,w_2)$ represents the velocity field of fluid (see Fig. 1).
\begin{figure}[ht]
\centering
\includegraphics[width=0.75\textwidth]{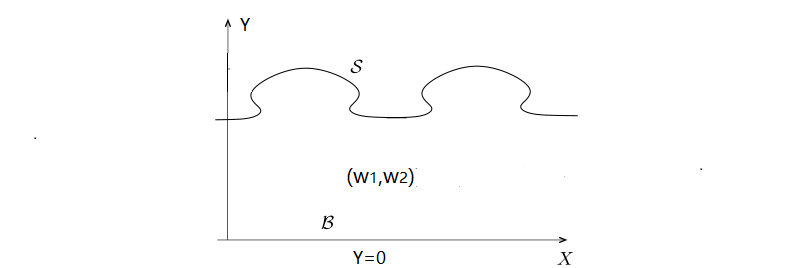}
\caption{The domain of fluid.}
\label{fig1}
\end{figure}\\
Here we add a factor of $\sqrt{\rho}$ to the typical definition of the stream function for incompressible fluid, which could capture the effects of stratification. It's known that $\rho$ is transported, then it must be constant on the streamlines. Thus we let streamline density function $\rho\in C^{2,\alpha}([0,|m|]; R^{+})$ be given by $\rho(X,Y)=\rho(-\psi(X,Y))$ throughout the fluid, where $m$ is the relative mass flux defined by
$$
m:=\int^{v(s)}_0\psi_Y(X,Y)dY, \quad s\in\mathbb{R}.
$$
It follows from Bernoulli's law that
\begin{eqnarray}
E=\frac{|\nabla\psi|^2}{2}+gY\rho+P, \label{eq1.2}
\end{eqnarray}
where $E$ is the hydraulic head and it is a constant along each streamline. In addition, there exists a function $\beta \in C^{1,\alpha}([0,|m|]; R)$ such that
\begin{eqnarray}
\frac{dE}{d\psi}=\beta(-\psi),  \label{eq1.3}
\end{eqnarray}
where $\beta(-\psi)$ is called the Bernoulli function. Indeed, the equation (\ref{eq1.3}) is known as the Yih-Long Equation (see \cite{Walsh09}). Physically it describes the variation of specific energy between the streamlines. Combining (\ref{eq1.2}) with (\ref{eq1.3}), we obtain that
\begin{eqnarray}
\frac{dE}{d\psi}=\Delta\psi-gY\rho'(-\psi)=\beta(-\psi(X,Y)). \label{eq1.4}
\end{eqnarray}
Since the surface tension is considered as a restoring force in this paper, then evaluating Bernoulli's law (\ref{eq1.2}) on the free surface $\mathcal{S}$ gives that
\begin{eqnarray}
2E=|\nabla\psi|^{2}+2gY\rho(-\psi)+2P_{atm}-2\sigma \mathcal{K}~~~on ~~\mathcal{S}, \label{eq1.5}
\end{eqnarray}
where $P_{atm}$ is the atmospheric pressure of the air, $\sigma$ is the coefficient of surface tension and $\mathcal{K}$ is the curvature of free surface $\mathcal{S}$.
Summarizing the considerations above, we obtain the following problem in the physical or $(X,Y)$ plane
\begin{eqnarray}
\left\{\begin{array}{ll}{ \Delta\psi=gY\rho'(-\psi)+\beta(-\psi)} & {\text { in }\Omega}, \\
{\psi=0} & {\text { on } \mathcal{S}}, \\
{\psi=-m} & {\text { on } \mathcal{B}},\end{array}\right. \label{eq1.6}
\end{eqnarray}
and
\begin{eqnarray}
|\nabla\psi|^{2}+2g\rho(-\psi)Y-2\sigma \mathcal{K}=Q \quad on \quad \mathcal{S},  \label{eq1.7}
\end{eqnarray}
where $Q=2(E|_{\mathcal{S}}-P_{atm})$ is the Bernoulli's constant. When the density function $\rho$ is reduced to a constant, then the right term of the first formula in (\ref{eq1.6}) would be $\beta(-\psi)=\sqrt{\rho}\gamma(-\psi)$ with $\gamma$ denoting the vorticity function of the homogeneous fluid.

It is obvious that the problem (\ref{eq1.6})-(\ref{eq1.7}) is difficult to handle due to the unknown surface $\mathcal{S}$. To reformulate the problem in terms of a harmonic function as in \cite{ConstantinSV,ConstantinV}, we need to remove all $\psi$ dependence on the right-hand side of the first equation in (\ref{eq1.6}). Then we will take the similar assumptions as in \cite{Haziot} that
\begin{equation}\label{eq1.8}
\rho(-\psi)=-A\psi+B, \quad \beta(-\psi)=\beta
\end{equation}
with the constants $A,\beta\in \mathbb{R}$ and $B\in \mathbb{R}^+$.
It's worth noting that the surface tension $\sigma$ was neglected in
\cite{Haziot}. However, we will see that some new phenomena would
occur when $\sigma>0$ in our paper. The assumption (\ref{eq1.8})
implies that the streamline density function $\rho(-\psi)$ depends
linearly on the pseudo-stream function $\psi$ and the variation
$\beta(-\psi)$ of energy $E$ along streamlines is homogeneous. As
$A=0$, the assumption $\beta(-\psi)=\beta$ is mathematically
equivalent to the case of constant vorticity. We also would like to
mention the interesting work \cite{HenryBM}, where they prove the
existence of steady periodic capillary-gravity stratified water
waves by using the naive flatting transform. Although the authors in
\cite{HenryBM} didn't require the specific form of stratification as
(\ref{eq1.8}), they added some technical restrictions (A1)-(A4) on
$f(Y,\psi)=gY\rho'(-\psi)+\beta(-\psi)$ such that their cases and
ours are mutually exclusive. In addition, the phenomena of
bifurcation from the two-dimensional kernel and the local stability
of bifurcation curves were not considered in \cite{HenryBM}.
However, these new features are our main focus in this paper.

Now let us talk briefly about the history of the problem. In the
previous century, most work continued to be both two-dimensional and
irrotational, where the stream function $\psi$ is harmonic in the
plane fluid region, which enjoys the advantage of being thoroughly
treatable by tools of complex analysis. However, as vorticity is
included (i.e. let $\rho\equiv1$ and $\sigma\equiv 0$ in
(\ref{eq1.6})-(\ref{eq1.7})), it was unclear for long time how to
leave the regime of small perturbations of configurations with a
flat surface. The breakthrough in this direction is due to
Constantin and Strauss \cite{ConstantinS} on the existence of
homogeneous gravity water waves of large amplitude. Building on
techniques in \cite{Constantin,ConstantinS}, as the density
$\rho=\rho(X,Y)$ and $\sigma\equiv0$ or the density $\rho\equiv1$
and $\sigma>0$, the rigourous results concerning the existence of
stratified flows or capillary-gravity waves were obtained via
bifurcation methods in \cite{Walsh09} and \cite{Wahlen06},
respectively.  Whereafter, the existence of water waves for
stratified flows with the additional complication of surface tension
was then addressed in \cite{Walsh14a,Walsh14b}.

All of these work mentioned are restricted to waves without
stagnation points. In this setting, one can use the so-called
hodograph transformation and express the problem as a quasilinear
elliptic problem. In order to allow for stagnation points and
critical layers, some papers
\cite{EhrnstromEW,HenryAM,HenryBM,Wahlen09} used a naive flattening
transform, where the vertical coordinate is scaled to a constant, to
establish the existence of corresponding water waves. However, a
drawback of this naive flattening is that it needs the surface
profile to be a graph which means the overhanging waves are
excluded. In fact, some numerical results
\cite{DyachenkoH1,DyachenkoH2} and recent mathematical studies
\cite{HurW} show that there are waves with overhanging surface
profiles for the constant vorticity. It is also reasonable to expect
that the profiles are overhanging for more general vorticity
distribution. Therefore, the authors in
\cite{ConstantinSV,ConstantinV} (or \cite{Martin,Matioc}) managed to
construct steady periodic gravity (or capillary-gravity) water waves
by considering the fluid domain as the image of a strip via a
conformal map.

In this paper, the generality of the flows we admit complicates the
analysis immensely, and the resulting governing equations which
emerge take the form of quasilinear pseudodifferential equation
caused by following the idea of conformal mapping due to
\cite{ConstantinV}. Based on this, the first contribution of this
paper is to establish the existence of steady periodic
capillary-gravity water waves of small amplitude for stratified
flows and further analyse its local stability.  The another novelty
here is that we find, as the Bernoulli's function $\beta$ is close
to some critical values, that the linearized problem would possess
two-dimensional kernel. For this, we apply a new version of
Crandall-Rabinowitz type local bifurcation theorem with multiple
parameters and kernels in \cite{Kielhofer} to obtain the two-model water waves.

The rest of this paper is arranged as follows. In Section 2, we
employ a suitable conformal map to transform the problem
(\ref{eq1.6})-(\ref{eq1.8}) into a quasilinear pseudodifferential
equation. Section 3 provides a brief proof of the existence of
small-amplitude stratified capillary-gravity waves, along with the
local stability analysis of these bifurcation curves. In Section 4,
we present a novel version of the Crandall-Rabinowitz type local
bifurcation theorem, which incorporates multiple kernels. Using the theorem, we derive the two-model water waves. Finally, in
Section 5, we demonstrate that the waves we have discovered exhibit
stagnation points within their interior regions.

\section{\bf Reformulation via a conformal mapping}

The main difficulties associated with the system (\ref{eq1.6})-(\ref{eq1.8}) are its nonlinear character and the fact that the interface $\mathcal{S}$ is unknown. Based on the fundamental work \cite{ConstantinV}, we are planing to introduce a suitable conformal mapping to transform the problem (\ref{eq1.6})-(\ref{eq1.8}) into a quasilinear pseudodifferential equation. To this end, let us define the following horizontal strips by
$$
\mathcal{R}_{h}:=\{ (x,y)\in \mathbb{R}^2: -h<y<0 \}.
$$
The main idea is to regard the unknown fluid domain $\Omega$ in $(X, Y)$-plane as the conformal image of the strip $\mathcal{R}_{h}$ in $(x,y)$-plane. For $2\pi$-periodic strip-like domain $\Omega$ of class $C^{2,\alpha}$ for $\alpha\in (0,1)$, whose boundary consists of $\mathcal{B}$ and a $2\pi$-periodic curve $\mathcal{S}$. It follows from \cite{ConstantinV} that there exists a positive constant $h=[v]$, where $[f]$ denotes the mean of function $f$ over one $2\pi$-period, such that we can find a conformal mapping $U+iV$ from $\mathcal{R}_{h}$ onto $\Omega$, which admits an extension as a homeomorphism of class $C^{2,\alpha}$ between the closure of the domain (see Fig. 2).
\begin{figure}[ht]
\centering
\includegraphics[width=0.75\textwidth]{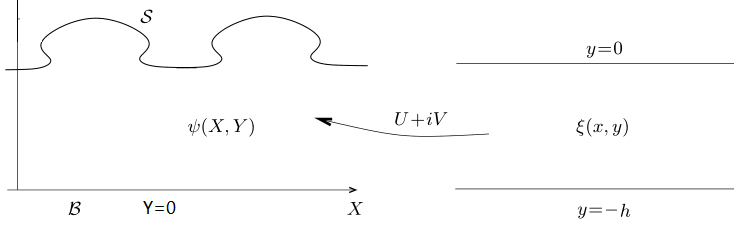}
\caption{The conformal parametrization of the fluid domain.}
\label{fig1}
\end{figure}

In addition, for $(x,y)\in\mathcal{R}_{h}$ there holds that
\begin{equation}\label{eq2.1}
\left\{\begin{array}{ll}
{U(x+2\pi,y)=U(x,y)+2\pi},\\
{V(x+2\pi,y)=V(x,y)},
\end{array}\right.
\end{equation}
with $U,V\in C^{2,\alpha}(\mathcal{R}_{h})$ satisfying
\begin{equation}\label{eq2.2}
\left\{\begin{array}{ll}
{U_x^2(x,0)+V_x^2(x,0)\neq0},\\
{x\mapsto(U(x,0),V(x,0))~\text{is injective}},
\end{array}\right.
\end{equation}
Moreover,
\begin{equation}\label{eq2.3}
\mathcal{S}=\{(U(x,0),V(x,0)):x\in\mathbb{R}\}.
\end{equation}
To understand conformal mapping $U+iV$, we define
$$
v(x):=V(x,0)\quad \forall x\in\mathbb{R}.
$$
Note that $V$ would be recovered uniquely from $v$ as the solution of
\begin{equation}\label{eq2.4}
\left\{\begin{array}{llll}
{\Delta V=0} &~in~{\mathcal{R}_{h}},\\
{V(x,0)=v(x)} & ~{x\in\mathbb{R}},\\
{V(x,-h)=0} &~{x\in\mathbb{R}}.\\
\end{array}\right.
\end{equation}
According to the Cauchy-Riemann formula, $U$ would also be determined by $v$. In addition, it follows from (\ref{eq2.4}) that (see \cite[Appendix A]{ConstantinSV} for more details)
\begin{equation}\label{eq2.5}
V_{y}(x,0)=\mathcal{G}_{h}(v)(x)
\end{equation}
for all $x\in \mathbb{R}$, where $w\mapsto\mathcal{G}_{h}(w)$ denotes the periodic Dirichlet-Neumann operator for the strip.

Since $\Omega$ is $2\pi$-periodic strip-like domain of class $C^{2,\alpha}$, then classical elliptic theory ensures that (\ref{eq1.6}) have a unique solution, which depends on $v$. If $\psi$ is the unique solution of (\ref{eq1.6}), we define $\xi: \mathcal{R}_{h}\rightarrow \mathbb{R}$ by
\begin{equation}\label{eq2.6}
\xi(x,y)=\psi(U(x,y),V(x,y)) \quad for \quad (x,y)\in\mathcal{R}_{h}.
\end{equation}
Based on the first formula of (\ref{eq1.6}) and the assumptions (\ref{eq1.8}), it's easy to deduce that
$$
(X,Y)\mapsto \psi(X,Y)-\frac{gA}{6}Y^3-\frac{C}{2}Y^2
$$
is a harmonic function in $\Omega$. Since harmonic functions are invariant under conformal mapping, it follows (\ref{eq2.6}) that
\begin{equation}\label{eq2.7}
\xi(x,y)-\frac{gA}{6}V^3-\frac{\beta}{2}V^2
\end{equation}
is harmonic in $\mathcal{R}_{h}$. Besides, it follows from the second and third formulas in (\ref{eq1.6}) that
\begin{equation}\label{eq2.8}
\left\{\begin{array}{ll}
{\xi(x,0)=0},\\
{\xi(x,-h)=-m},\\
\end{array}\right.
\end{equation}
By using the chain rule and Cauchy-Riemann equations, we easily obtain that
\begin{equation}\label{eq2.9}
\xi_{x}^2+\xi_{y}^2=\left( \psi_{X}^2(U,V)+ \psi_{Y}^2(U,V)\right)\left( V_{x}^2+ V_{y}^2\right) ~~\text{in} ~~\mathcal{R}_{h},
\end{equation}
Let $\eta: \mathcal{R}_{h}\rightarrow \mathbb{R}$ be given by
$$
\eta=\xi-\frac{gA}{6}V^3-\frac{\beta}{2}V^2+m,
$$
then it follows from (\ref{eq2.7})-(\ref{eq2.8}) that
\begin{equation}\label{eq2.10}
\left\{\begin{array}{llll}
{\Delta \eta=0} &~in~{\mathcal{R}_{h}},\\
{\eta(x,0)=m-\frac{gA}{6}v^3(x)-\frac{\beta}{2}v^2(x)} & ~{x\in\mathbb{R}},\\
{\eta(x,-h)=0} &~{x\in\mathbb{R}},\\
\end{array}\right.
\end{equation}
Similar as (\ref{eq2.5}), the system (\ref{eq2.10}) would give that
\begin{equation}\label{eq2.11}
\eta_{y}(x,0)=\mathcal{G}_{h}\left(m-\frac{gA}{6}v^3-\frac{\beta}{2}v^2\right)(x)
\end{equation}
for all $x\in \mathbb{R}$ where $w\mapsto\mathcal{G}_{h}(w)$ denotes the periodic Dirichlet-Neumann operator for the strip.

On the other hand, we can deduce that the interface equation (\ref{eq1.7}) on $y=0$ would be reduced as
\begin{equation}\label{eq2.12}
\frac{\left(\eta_{y}+\frac{gA}{2}V^{2}V_{y}+\beta VV_{y}\right)^2}{V_{x}^2+V_{y}^2}=Q-2gBV+2\sigma\mathcal{K}
\end{equation}
by using (\ref{eq2.9}).
Since
$$
\mathcal{G}_{d}(u)=\frac{[u]}{d}+(\mathcal{C}(u-[u]))'=\frac{[u]}{d}+\mathcal{C}_{d}(u')
$$
and the periodic Hilbert transform for the strip is given by
$$
(\mathcal{C}_{d}(u))(x)=\sum_{n=1}^{\infty}a_{n}\coth(nd)\sin(nx)-\sum_{n=1}^{\infty}b_{n}\coth(nd)\cos(nx)
$$
for all $u\in L^{2}_{2\pi}$ with [u]=0, where $u$ has the Fourier series expansion
$$
u(x)=\sum^{\infty}_{n=1}a_{n}\cos(nx)+\sum^{\infty}_{n=1}b_{n}\sin(nx),
$$
then $(\ref{eq2.5})$ and $(\ref{eq2.11})$ would become
\begin{equation}\label{eq2.13}
\left\{\begin{array}{ll}
{V_{y}(x,0)=1+\mathcal{C}_{h}(v')},\\
{\eta_{y}(x,0)=\frac{m}{h}-\frac{gA[v^{3}]}{6h}-\frac{gA}{2}\mathcal{C}_{h}(v^{2}v')-\frac{\beta[v^2]}{2h}-\beta\mathcal{C}_{h}(vv')}.
\end{array}\right.
\end{equation}
It follows from the first equation in (\ref{eq2.13}) and Cauchy-Riemann formula that
$$U_{x}(x,0)=V_{y}(x,0)=1+\mathcal{C}_{h}(v'),$$
which means that $U(x,0)=x+\mathcal{C}_{h}(v-h)(x)$ up to a constant. Thus, the curvature $\mathcal{K}$ of the free surface $\mathcal{S}$ can be shown as
\begin{equation}\label{eq2.14}
\mathcal{K}=\frac{U'(x,0)V''(x,0)-V'(x,0)U''(x,0)}{(U'^{2}(x,0)+V'^{2}(x,0))^{\frac{3}{2}}}=
\frac{v''+\mathcal{C}_{h}(v')v''-v'\mathcal{C}_{h}(v'')}{((1+\mathcal{C}_{h}(v'))^2+v'^{2})^{\frac{3}{2}}}.
\end{equation}

Taking (\ref{eq2.13}) and (\ref{eq2.14}) into (\ref{eq2.12}), we finally reduce the problem (\ref{eq1.6})-(\ref{eq1.7}) to the following quasilinear pseudodifferential equation on the function $v\in C^{2,\alpha}_{2\pi}(\mathbb{R})$ which satisfy
\begin{eqnarray}\label{eq2.15}
& & \left( \frac{m}{h}-\frac{gA[v^3]}{6h}-\frac{gA}{2}\mathcal{C}_{h}(v^2v')-\frac{\beta[v^2]}{2h}-\beta\mathcal{C}_{h}(vv')+\frac{gA}{2}v^2(1+\mathcal{C}_{h}(v'))+\beta v (1+\mathcal{C}_{h}(v')) \right)^2  \nonumber\\
& &=(Q-2gBv)\left( v'^2+(1+\mathcal{C}_{h}(v'))^2\right)+2\sigma\frac{v''+\mathcal{C}_{h}(v')v''-v'\mathcal{C}_{h}(v'')}{((1+\mathcal{C}_{h}(v'))^2+v'^{2})^{\frac{1}{2}}}
\end{eqnarray}
with
\begin{equation*}
\left\{\begin{array}{llll}
{[v]=h},\\
{v(x)>0,~~\text{for all}~~ x\in\mathbb{R}},\\
{\text{the mapping}~~ x\mapsto( x+\mathcal{C}_{h}(v-h)(x),v(x))~~ \text{is injective on}~~ \mathbb{R}},\\
{v'(x)^2+(1+\mathcal{C}_{h}(v')(x))^2\neq0,~~ \text{for all}~~ x\in\mathbb{R}},\\
\end{array}\right.
\end{equation*}
where $C^{k,\alpha}_{2\pi}(\mathbb{R})$ denotes the space of $2\pi$-periodic functions whose partial derivatives up to order $k$ are H\"{o}lder continuous with exponent $\alpha$ over their domain of definition.



\section{Local bifurcation from a simple eigenvalue and stability}
In this section, we will first show the existence of solutions $(m,Q,v)$ of (\ref{eq2.15}) in the space $\mathbb{R}\times\mathbb{R}\times C^{2,\alpha}_{2\pi,e}(\mathbb{R})$ with $[v]=h$, where
$$
C^{2,\alpha}_{2\pi,e}(\mathbb{R}):=\{f\in C^{2,\alpha}_{2\pi}(\mathbb{R}): f(x)=f(-x), \forall x\in \mathbb{R}\}
$$
by using the remarkable Crandall-Rabnowitz local bifurcation theorem in \cite{CrandallR}. Furthermore, we will conclude the stability of these local solutions curves via the exchange of stability theorem due to Crandall-Rabnowitz \cite{CrandallR1}. We stated the theorems in appendix where we also explained the concept of stability.
\subsection{\bf The existence}

At this time, we only need to verify the assumptions $(H1)$ and $(H2)$ in Theorem \ref{thm6.1}.
Given $v\in C^{2,\alpha}_{2\pi,e}(\mathbb{R})$ with $[v]=h$, it is natural to set
$$
v=w+h.
$$
Then we would have a family of trivial solutions to (\ref{eq2.15}) for which $v=h$ (i.e $w=0$) if and only if $Q$ and $m$ are related by
$$
Q=2gBh+\left( \frac{m}{h}+\frac{\beta h}{2}+\frac{gAh^2}{3} \right)^2
$$
for $m\in \mathbb{R}$.
For convenience, we introduce the parameters
$$
\lambda:=\frac{m}{h}+\frac{\beta h}{2}+\frac{gAh^2}{3},\quad \mu:=Q-2gBh-\lambda^2.
$$
Then we can rewrite the problem (\ref{eq2.15}) as the following bifurcation problem
\begin{eqnarray}\label{eq3.1}
& & F(\lambda,(\mu,w))  \\ \nonumber
& & :=\left( \lambda-\beta\left( \frac{[w^2]}{2h}-w+\mathcal{C}_{h}(ww')-w\mathcal{C}_{h}(w') \right) \right.\\ \nonumber
& &\left.-\frac{gA}{2}\left(\frac{[w^3]}{3h}-[w^2]-w^2-2hw+\mathcal{C}_{h}(w^2w')-w^2\mathcal{C}_{h}(w')+2h\mathcal{C}_{h}(ww')-2hw\mathcal{C}_{h}(w') \right) \right)^2 \\ \nonumber
& &-\left( \lambda^2+\mu-2gBw \right)\left(w'^{2}+\left(1+\mathcal{C}_{h}(w')\right)^2\right)-2\sigma\frac{w''+\mathcal{C}_{h}(w')w''
-w'\mathcal{C}_{h}(w'')}{((1+\mathcal{C}_{h}(w'))^2+w'^{2})^{\frac{1}{2}}}=0
\end{eqnarray}
for $\lambda,\mu\in \mathbb{R}$ and $w\in C^{2,\alpha}_{2\pi}(\mathbb{R})$ and it is obvious that
$$F(\lambda,(0,0))=0,$$
which ensures (H1) in Theorem \ref{thm6.1} holds naturally. In fact, this family represents a curve
$$
\mathcal{K}_{triv}=\{ \left(\lambda, (0,0)\right) : \lambda\in \mathbb{R} \}
$$
in the space $\mathbb{R}\times\mathbb{R}\times C^{2,\alpha}_{2\pi,e}(\mathbb{R})$. These solutions are the laminar flows in the fluid domain, which is bounded below by $\mathcal{B}$ and is bounded above by $Y\equiv h$ with streamfunction
\begin{eqnarray}\label{eq3.2}
\psi(X,Y)=\frac{\beta}{2}Y^2+\frac{gA}{6}Y^3+\left( \frac{m}{h}-\frac{\beta h}{2}-\frac{gAh^2}{6} \right)Y-m.
\end{eqnarray}

In the following, it remains for us to verify (H2) in Theorem 5 by choosing the Banach spaces
$$
X:=C^{2,\alpha}_{2\pi,e}(\mathbb{R}),\quad Y:=C^{0,\alpha}_{2\pi,e}(\mathbb{R}).
$$
It follows from (\ref{eq3.1}) that the mapping
$$
F: \mathbb{R}\times\mathbb{R}\times  X\rightarrow Y
$$
is real-analytic. Then for $(\nu, \phi)\in \mathbb{R}\times X$ we can compute that
\begin{eqnarray}\label{eq3.3}
~&~&F_{(\mu,w)}(\lambda, (0,0))(\nu, \phi)=\nu F_{\mu}(\lambda,(0,0))+F_{w}(\lambda,(0,0))\phi\nonumber \\
&=&-\nu+2\lambda^2\mathcal{C}_{h}(\phi')-2\left(\beta\lambda+gAh\lambda+gB\right)\phi-2\sigma\phi''.
\end{eqnarray}
Since the function $\phi$ is even, $2\pi$ periodic and has zero average, then we expand it into the following Fourier series
\begin{eqnarray}\label{eq3.4}
\phi=\sum_{n=1}^{\infty}a_n\cos(nx).
\end{eqnarray}
Taking (\ref{eq3.4}) into (\ref{eq3.3}), we can obtain that
\begin{equation}\label{eq3.5}
F_{(\mu,w)}(\lambda,(0,0))\left(\nu, \sum_{n=1}^{\infty}a_{n}\cos(nx)\right)=-\nu+\sum_{n=1}^{\infty} D(n,\lambda) a_{n}\cos(nx),
\end{equation}
whereby for $T_n:=\frac{\tanh(nh)}{n}$
$$
D(n,\lambda)=-2\left( \frac{\lambda^2}{T_n}-\beta\lambda-gAh\lambda-(gB+\sigma n^2) \right).
$$
Indeed, $D(n,\lambda)$ is called dispersion relation. Once the dispersion relation
vanishes, which means the linearized operator $F_{(\mu,w)}(\lambda,(0,0))$ is degenerate and
the nontrivial solutions may occur. That is to say, the bifurcation points $\lambda$ may  be
\begin{equation}\label{eq3.6}
\lambda_{n,\pm}^*:=\frac{\beta+Agh}{2}T_n\pm\sqrt{\frac{(\beta+Agh)^2}{4}T_{n}^2+(gB+\sigma n^2)T_{n}}
\end{equation}
for any integer $n\geq 1$.
From (\ref{eq3.6}), it is easy to see that
$$
\lambda_{n,+}^*>0 \quad\text{and} \quad \lambda_{n,-}^*<0.
$$

Now we claim that (H2) holds for every $\lambda_*\in\{\lambda_{n,\pm}^*: n\in \mathbb{N}\}$. For any such $\lambda_*$, it is not difficult to find that $\mathcal{N}(F_{(\mu,w)}(\lambda,(0,0)))$ is one-dimensional and generated by $(0,w^*)\in \mathbb{R}\times X$, where
$$
w^*=\cos(nx) \quad \text{for} \quad x\in\mathbb{R}.
$$
Note that the range $\mathcal{R}(F_{(\mu,w)}(\lambda,(0,0)))$ is the closed subset of $Y$ consisting of the elements $\varphi$ with
$$
\int_{-\pi}^{\pi}\varphi\cos(nx)dx=0.
$$
Thus, the co-range $Y\setminus \mathcal{R}(F_{(\mu,w)}(\lambda,(0,0)))$ is generated by $\{w^*\}$ which is also one-dimensional.
Finally, we can check easily that the tranversality condition holds by computing that
$$
F_{\lambda(\mu,w)}(\lambda_*, (0,0))(1,(0,w^*))= -2\left(\frac{2\lambda_*}{T_n}-\beta-Agh\right)w^*\notin \mathcal{R}(F_{(\mu,w)}(\lambda,(0,0)))
$$
due to $\left(\frac{2\lambda_*}{T_n}-\beta-Agh\right)\neq 0$ by observing (\ref{eq3.6}).
Therefore, we can obtain the following result by applying the Theorem \ref{thm6.1} in Appendix.

\begin{theorem} \label{thm3.1} (Local bifurcation from one-dimensional kernel)
For Bernoulli's function $\beta\in \mathbb{R}$, let $\lambda_{n,\pm}^*$ be shown as in (\ref{eq3.6}). For any $\lambda\in \mathbb{R}\backslash \left\{\lambda_{n,\pm}^*:n\in \mathbb{N}^+\right\}$, the
solutions of (\ref{eq3.1}) are trivial. For $\lambda\in \lambda_{n,\pm}^*$ and each choice of sign $\pm$, there exists in the space $\mathbb{R}\times X$ a
continuous solutions curve
$$\mathcal{K}_{n,\pm} = \left\{\left(\lambda_{n,\pm}(s), (\mu(s),w(s))\right) : s\in \mathbb{R}\right\}$$
of (\ref{eq3.1}) satisfying

(i) $\left(\lambda_{n,\pm}(0), (\mu(0),w(0))\right)=\left(\lambda_{n,\pm}^*, (0,0)\right)$;

(ii) $w(s)= sw^*+o(s)$ for $0\leq|s|<\varepsilon$, where $w^*=\cos(nx)$ and $\varepsilon>0$ sufficiently
small;

(iii) there exist a neighbourhood $\mathcal{U}_{n,\pm}$ of $\left(\lambda_{n,\pm}^*, (0, 0)\right)$ in $\mathbb{R}\times X$ and
$\varepsilon>0$ sufficiently small such that
\begin{equation}
\left\{\left(\lambda, (\mu,w)\right)\in \mathcal{U}_{n,\pm} :  F(\lambda,(\mu,w))=0\right\} = \left\{\left(\lambda_{n,\pm}(s), (\mu(s),w(s))\right) : 0 \leq |s| < \varepsilon\right\}.\nonumber
\end{equation}
\end{theorem}

\subsection{\bf The formal stability}
Inspired by \cite{ConstantinS1}, we will establish the following formal stability in the linearized sense of the curves of solutions $\mathcal{K}_{n,\pm}$ obtained in Theorem \ref{thm3.1}. We would like to mention that the concept of formal stability has been given in Appendix.

\begin{theorem} \label{thm4.1} (The stability of local bifurcation curve)\\
(i) For the bifurcation branch $(\lambda(s),w(s))$ of (\ref{eq3.1}) emanating from $\lambda_{n,\pm}^*$, we have that
$$\lambda_{n,\pm}'(0)=0$$
 and there exists a small positive constant $\varepsilon$ such that when $\beta,\sigma\in[-\varepsilon,+\varepsilon]$ and $A\geq 0$, there holds that
\begin{eqnarray}
\lambda_{n,+}''(0)>0,\nonumber
\end{eqnarray}
and when $\beta,\sigma\in[-\varepsilon,+\varepsilon]$ and $A\leq 0$, there holds that
\begin{eqnarray}
\lambda_{n,-}''(0)<0.\nonumber
\end{eqnarray}
(ii) As $\beta,\sigma\in[-\varepsilon,+\varepsilon]$ and $A\geq 0$, near the bifurcation point $\lambda_{n,+}^*$, the laminar solution is stable formally for $\lambda>\lambda_{n,+}$ but unstable for $\lambda<\lambda_{n,+}^*$. As $\beta,\sigma\in[-\varepsilon,+\varepsilon]$ and $A\leq 0$, near the bifurcation point $\lambda_{n,-}^*$, the laminar solution is unstable for $\lambda>\lambda_{n,-}^*$ but stable formally for $\lambda<\lambda_{n,-}^*$.\\
(iii) As $\beta,\sigma\in[-\varepsilon,+\varepsilon]$ and $A\geq 0$, the nontrivial solutions are unstable near the bifurcation points $\lambda_{n,+}^*$ and as
 $\beta,\sigma\in[-\varepsilon,+\varepsilon]$ and $A\leq 0$, the nontrivial solutions are unstable near the bifurcation points $\lambda_{n,-}^*$. (see Fig.3)
\end{theorem}
\begin{figure}[ht]
\centering
\includegraphics[width=0.75\textwidth]{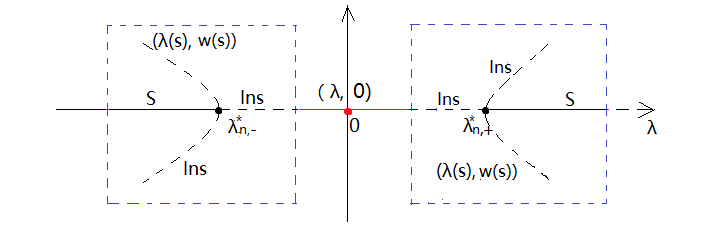}
\centering
\caption{The local stability near bifurcation points $\lambda_{n,\pm}^*$.}
\end{figure}
\begin{proof}
Let us first use the Theorem \ref{thm6.2} to prove the bifurcation direction (i).
For any $w\in C^{2,\alpha}_{2\pi,e}(\mathbb{R})$, define
\begin{eqnarray}
\widetilde{l}(w)=\frac{1}{\pi}\int_{-\pi}^{\pi} w\cos(nx) dx.\nonumber
\end{eqnarray}
Then $\widetilde{l}$ is a linear functional on $C^{2,\alpha}_{2\pi,e}(\mathbb{R})$ such that $\widetilde{l}\left(\cos(nx)\right)=1$.
It is clearly $\mathcal{R}(F_{(\mu,w)}\left(\lambda^*_{n,\pm},\beta,(0,0)\right))=\left\{u:\widetilde{l}(u)=0\right\}$.
Let $w^*=\cos(nx)$, then it follows from (\ref{eq3.5}) and (\ref{eq3.6}) that
\begin{equation}\label{eq3.7}
\left\langle \widetilde{l}, F_{\lambda (\mu,w)}\left(\lambda_{n,\pm}^*,\beta,(0,0)\right)(1,(0,w^*))\right\rangle=\mp2\sqrt{(\beta+Agh)^2+\frac{4(gB+\sigma n^2)}{T_n}}.
\end{equation}
For any $\lambda_*\in \{\lambda^*_{n,\pm}\}$, it follows from (\ref{eq3.1}) that
\begin{eqnarray}\label{eq3.8}
&~&F_{(\mu,w)^2}(\lambda_*,\beta,(0,0))((0,w^*),(0,w^*)) \nonumber \\
&=&2\left(\beta+gAh\right)^2w^{*2}-2\lambda_*\beta\left( \frac{[w^{*2}]}{h}+2\mathcal{C}_{h}(w^*w^{*'}) -2w^*\mathcal{C}_{h}(w^{*'}) \right) \nonumber \\
&+&\lambda_*gA\left( 2[w^{*2}]+2w^{*2}+4hw^{*}\mathcal{C}_{h}(w^{*'})-4h\mathcal{C}_{h}(w^*w^{*'}) \right)+8gBw^*\mathcal{C}_{h}(w^{*'}) \nonumber \\
&-&2\lambda_*^2\left( w^{*'2}+\mathcal{C}_{h}^{2}(w^{*'}) \right)+4\sigma w^{*'}\mathcal{C}_{h}(w^{*''})
\end{eqnarray}
which gives that
\begin{eqnarray}\label{eq3.9}
\left\langle \widetilde{l}, F_{(\mu,w)^2}(\lambda_*,\beta,(0,0))((0,w^*),(0,w^*))\right\rangle=0
\end{eqnarray}
by using the facts
$$
\mathcal{C}_{h}(\cos(nx))=\coth(nh)\sin(nx), ~~\mathcal{C}_{h}(\sin(nx))=-\coth(nh)\cos(nx)
$$
and
$$
\int^{\pi}_{-\pi} \cos^3(nx)dx=0, \quad \int^{\pi}_{-\pi} \sin^2(nx) \cos(nx)dx=0.
$$
Therefore, combining (\ref{eq3.7}) with (\ref{eq3.9}) and applying Theorem \ref{thm6.2}, we obtain that
\begin{equation}
\lambda_{n,\pm}'(0)=-\frac{\left\langle \widetilde{l}, F_{(\mu,w)^2}(\lambda_*,\beta,(0,0))((0,w^*),(0,w^*))\right\rangle}{2\left\langle \widetilde{l}, F_{\lambda (\mu,w)}\left(\lambda_*,\beta,(0,0)\right)(1,(0,w^*)\right\rangle} =0,\nonumber
\end{equation}
which indicates that the bifurcation is pitchfork bifurcation.

By a complex calculation, we further obtain that
\begin{eqnarray}\label{eq3.10}
&~&F_{(\mu,w)^3}(\lambda_*,\beta,(0,0))((0,w^*),(0,w^*),(0,w^*)) \nonumber \\
&=&6\beta(\beta+gAh)w^*\left( 2w^*\mathcal{C}_{h}(w^{*'})-2\mathcal{C}_{h}(w^*w^{*'})-\frac{[w^{*2}]}{h} \right) \nonumber \\
&+&6gA(\beta+gAh)w^*\left( [w^{*2}]+w^{*2}-2h\mathcal{C}_{h}(w^*w^{*'})+2hw^*\mathcal{C}_{h}(w^{*'}) \right) \nonumber \\
&-& 2\lambda_*gA\left( \frac{[w^{*3}]}{h}+3\mathcal{C}_{h}(w^{*2}w^{*'})-3w^{*2}\mathcal{C}_{h}(w^{*'}) \right)+12gBw^*\left(w^{*'2}+\mathcal{C}_{h}^2(w^{*'})\right) \nonumber \\
&-&2\sigma\left( 10w^{*'}\mathcal{C}_{h}(w^{*'})\mathcal{C}_{h}(w^{*''})-3w^{*''}\mathcal{C}_{h}^2(w^{*'})+w^{*''}-w^{*''}w^{*'2}\right).\nonumber \\
\end{eqnarray}
Based on (\ref{eq3.10}), we obtain that
\begin{eqnarray}\label{eq3.11}
&~&\left\langle \widetilde{l},F_{(\mu,w)^3}(\lambda_*,\beta,(0,0))((0,w^*),(0,w^*),(0,w^*))\right\rangle  \nonumber \\
&=& 3\beta(\beta+gAh)\left( 3n\coth(nh)-n\coth(2nh)-\frac{1}{h} \right) \nonumber \\
&+& 3gA(\beta+gAh)\left( \frac{5}{2}+3hn\coth(nh)-hn\coth(2nh) \right) \nonumber \\
&+& 2\lambda_*gAn\coth(nh)+3gB\left( n^2+3n^2\coth^2(nh) \right) \nonumber \\
&-& \frac{\sigma}{2}\left( 19n^4\coth^2(nh)+n^4-4n^2 \right).
\end{eqnarray}
If $\beta=0$ and $\sigma=0$, it follows from (\ref{eq3.11}) that
\begin{eqnarray}\label{eq3.12}
&~&\left\langle \widetilde{l},F_{(\mu,w)^3}(\lambda_*,\beta,(0,0))((0,w^*),(0,w^*),(0,w^*))\right\rangle  \nonumber \\
&=& 3g^2A^2h\left( \frac{5}{2}+3hn\coth(nh)-hn\coth(2nh) \right) \nonumber \\
&+& 2\lambda_*Agn\coth(nh)+3gB\left( n^2+3n^2\coth^2(nh) \right), \nonumber
\end{eqnarray}
which implies that
$$
\left\langle \widetilde{l},F_{(\mu,w)^3}(\lambda_{n,+}^*,0,(0,0))((0,w^*),(0,w^*),(0,w^*))\right\rangle>0, \quad \text{for} \quad A\geq 0
$$
and
$$
\left\langle \widetilde{l},F_{(\mu,w)^3}(\lambda_{n,-}^*,0,(0,0))((0,w^*),(0,w^*),(0,w^*))\right\rangle>0, \quad \text{for} \quad A\leq 0.
$$

By continuity for small $\varepsilon$, if $\beta,\sigma\in[-\varepsilon,+\varepsilon]$ and $A\geq 0$, we have that
$$\left\langle \widetilde{l},F_{(\mu,w)^3}(\lambda_{n,+}^*,\beta,(0,0))((0,w^*),(0,w^*),(0,w^*))\right\rangle>0,$$
and if $\beta,\sigma\in[-\varepsilon,+\varepsilon]$ and $A\leq 0$, we have that
$$\left\langle \widetilde{l},F_{(\mu,w)^3}(\lambda_{n,-}^*,\beta,(0,0))((0,w^*),(0,w^*),(0,w^*))\right\rangle>0.$$

Therefore, as $\beta,\sigma\in[-\varepsilon,+\varepsilon]$ and $A\geq 0$, applying Theorem \ref{thm6.2} again, we have that
\begin{eqnarray}\label{eq3.12}
\lambda_{n,+}''(0)&=&\frac{\left\langle \widetilde{l},F_{(\mu,w)^3}(\lambda_{n,+}^*,\beta,(0,0))((0,w^*),(0,w^*),(0,w^*))\right\rangle}{-3\left\langle \widetilde{l}, F_{\lambda (\mu,w)}\left(\lambda_{n,+}^*,\beta,(0,0)\right)(1,(0,w^*))\right\rangle}
>0
\end{eqnarray}
and as $\beta,\sigma\in[-\varepsilon,+\varepsilon]$ and $A\leq 0$, we have that
\begin{eqnarray}\label{eq3.13}
\lambda_{n,-}''(0)&=&\frac{\left\langle \widetilde{l},F_{(\mu,w)^3}(\lambda_{n,-}^*,\beta,(0,0))((0,w^*),(0,w^*),(0,w^*))\right\rangle}{-3\left\langle \widetilde{l}, F_{\lambda (\mu,w)}\left(\lambda_{n,-}^*,\beta,(0,0)\right)(1,(0,w^*))\right\rangle}
<0,
\end{eqnarray}
which indicates that the bifurcation is supercritical at bifurcation point $\lambda_{n,+}^*$ and subcritical at bifurcation point $\lambda_{n,-}^*$, respectively.

In order to prove (ii) and (iii), we need to apply the exchange of stability Theorem \ref{thm6.3}. Based on the arguments above, we have known that $0$ is a simple eigenvalue of $T:=F_{(\mu,w)}\left(\lambda_{n,\pm}^*,\beta,(0,0)\right)$ and the transversal condition is satisfied.
Taking $K$ equals the identical operator, since $0$ is a simple eigenvalue of $T$, we see that $0$ is a $K$ simple eigenvalue of $T$.
So all assumptions of Theorem \ref{thm6.3} are satisfied.

Applying Theorem \ref{thm6.3} by letting $\psi^*=(0,\cos(nx))$, there are eigenvalues $\theta(s)$, $\kappa(\lambda_{n,\pm})\in \mathbb{R}$ with eigenvectors $(0,u(s))$,
$(0,\psi(\lambda_{n,\pm}))\in X $, such that
\begin{eqnarray}
F_{(\mu,w)}(\lambda_{n,\pm}(s),\beta,(\mu,w(s)))(0,u(s))=\theta_{\pm}(s)u(s),\nonumber
\end{eqnarray}
\begin{eqnarray}
F_{(\mu,w)}(\lambda_{n,\pm},\beta,(0,0))(0,\psi(\lambda_{n,\pm}))=\kappa(\lambda_{n,\pm})\psi(\lambda_{n,\pm})\nonumber
\end{eqnarray}
with
\begin{eqnarray} \label{eq3.14}
\theta(0)=\kappa\left(\lambda^*_{n,\pm}\right)=0,~~u(0)=\psi\left(\lambda^*_{n,\pm}\right)=\cos(nx)
\end{eqnarray}
and
\begin{eqnarray} \label{eq3.15}
\kappa'\left(\lambda^*_{n,\pm}\right)\neq 0,~~\lim_{s\rightarrow 0,\kappa(s)\neq 0}\frac{s\lambda_{n,\pm}'(s)}{\theta_{\pm}} =-\frac{1}{\kappa'\left(\lambda^*_{n,\pm}\right)}.
\end{eqnarray}

It follows from (\ref{eq3.5}) that
\begin{eqnarray}
F_(\mu,w)(\lambda_{n,\pm},\beta,(0,0))(0,\cos(nx))&=&D(n,\lambda_{n,\pm},\beta)\cos(nx).\nonumber
\end{eqnarray}
So we have that
\begin{eqnarray} \label{eq3.16}
\kappa'\left(\lambda^*_{n,\pm}\right)= -2\left( \frac{2\lambda^*_{n,\pm}}{T_n}-\beta-gAh \right).
\end{eqnarray}
When $n\geq1$, we deduce from (\ref{eq3.6}) that
\begin{eqnarray}
\kappa'\left(\lambda^*_{n,+}\right)= -2\left( \frac{2\lambda^*_{n,+}}{T_n}-\beta-gAh \right)=-2\sqrt{(\beta+Agh)^2+\frac{4}{T_n}\left( gB+\sigma n^2 \right)}<0 \nonumber
\end{eqnarray}
and
\begin{eqnarray}
\kappa'\left(\lambda^*_{n,-}\right)= -2\left( \frac{2\lambda^*_{n,-}}{T_n}-\beta-gAh \right)=
2\sqrt{(\beta+Agh)^2+\frac{4}{T_n}\left( gB+\sigma n^2 \right)}>0. \nonumber
\end{eqnarray}
Considering $\kappa\left(\lambda^*_{n,\pm}\right)=0$ in the (\ref{eq3.14}), then we can conclude that $\kappa\left(\lambda\right)<0$ if $\lambda>\lambda^*_{n,+}$ and $\kappa\left(\lambda\right)>0$ if $\lambda<\lambda^*_{n,+}$ for $\left\vert \lambda-\lambda^*_{n,+}\right\vert$ sufficiently small. Similarly, $\kappa\left(\lambda\right)>0$ if $\lambda>\lambda^*_{n,-}$ and $\kappa\left(\lambda\right)<0$ if $\lambda<\lambda^*_{n,-}$ for $\left\vert \lambda-\lambda^*_{n,-}\right\vert$ sufficiently small. This implies near the bifurcation point $\lambda^*_{n,+}$ that the laminar solution is stable formally for $\lambda>\lambda^*_{n,+}$ but unstable for $\lambda<\lambda^*_{n,+}$. Similarly, near the bifurcation point $\lambda^*_{n,-}$, the laminar solution is unstable for $\lambda>\lambda^*_{n,-}$ but stable formally for $\lambda<\lambda^*_{n,-}$.

On the other hand, by using (\ref{eq3.15}) and (\ref{eq3.16}), we have that
\begin{eqnarray}\label{eq3.17}
2\left( \frac{2\lambda^*_{n,\pm}}{T_n}-\beta-gAh \right)\lim_{s\rightarrow 0,\beta(s)\neq 0}\frac{s\lambda_{n,\pm}'(s)}{\theta_{\pm}(s)} =1.
\end{eqnarray}
It follows from (\ref{eq3.17}) that
\begin{eqnarray}\label{eq3.18}
\lim_{s\rightarrow 0,\beta(s)\neq 0}\frac{s\lambda_{n,+}'(s)}{\theta_{+}(s)} =\frac{1}{2\sqrt{(\beta+Agh)^2+\frac{4}{T_n}\left( gB+\sigma n^2 \right)}},
\end{eqnarray}
\begin{eqnarray}\label{eq3.19}
\lim_{s\rightarrow 0,\beta(s)\neq 0}\frac{s\lambda_{n,-}'(s)}{\theta_{-}(s)} =-\frac{1}{2\sqrt{(\beta+Agh)^2+\frac{4}{T_n}\left( gB+\sigma n^2 \right)}}.
\end{eqnarray}

We have proved that $\lambda_{n,\pm}'(0)=0$. Then by the Taylor expansion of $\lambda_{n,\pm}'(s)$ at $s=0$ we can write
\begin{eqnarray} \label{eq3.20}
\lambda_{n,\pm}'(s)=s \lambda_{n,\pm}''(0)+o\left(s\right).
\end{eqnarray}
Thus it follows (\ref{eq3.18})-(\ref{eq3.20}) that
\begin{eqnarray}
\lim_{s\rightarrow 0}\frac{s^2\lambda_{n,+}''(0)+o\left(s^2\right)}{\theta_{+}(s)}=\frac{1}{2\sqrt{(\beta+Agh)^2+\frac{4}{T_n}\left( gB+\sigma n^2 \right)}}>0\nonumber
\end{eqnarray}
and
\begin{eqnarray}
\lim_{s\rightarrow 0}\frac{s^2\lambda_{n,-}''(0)+o\left(s^2\right)}{\theta_{-}(s)}=-\frac{1}{2\sqrt{(\beta+Agh)^2+\frac{4}{T_n}\left( gB+\sigma n^2 \right)}}<0\nonumber
\end{eqnarray}
Furthermore, by using (\ref{eq3.12}) and (\ref{eq3.13}), we obtain that
\begin{eqnarray}
\theta_{+}(s)>0,\qquad \theta_{-}(s)>0.\nonumber
\end{eqnarray}
for $\left|s\right|$ ($\neq0$) small enough. This implies that the nontrivial solutions are unstable near the bifurcation points $\lambda_{n,\pm}^*$.
\end{proof}

\begin{remark}
It is worth noting that the stratified form (i.e. the sign of $A$) has an essential effect on local stability of two bifurcation curves $\mathcal{K}_{n,\pm}$. In particular, if taking $A=0$, then we obtain the local stability of capillary-gravity water waves. The local stability of gravity water waves also follows by further taking $\sigma=0$.
\end{remark}

\section{Local bifurcation from a double eigenvalue}
In section 3, the local bifurcation problem of (\ref{eq3.1}) was studied in the case when the kernel is one-dimensional. There we only choose $\lambda$ as the bifurcation parameter by fixing the Bernoulli's function $\beta$. In this section, we will regard $\beta$ as another parameter. Moreover, we find that if $(\beta+Agh)^2$ is sufficiently close to the constant $\beta_{m,n}$ defined by (\ref{eq4.2}), then the resonance phenomenon would occur.
That is to say, the kernel space of the linearized operator $F_{(\mu,w)}(\lambda_{n,\pm}^*, \beta^*_\pm, (0,0))$ is two-dimensional at this time.
Thus, the classical Crandall-Rabinowitz local bifurcation Theorem \ref{thm6.1} is not available.
For this, we would apply the local bifurcation theorem with multiple kernels and parameters due to Kielh\"{o}fer to obtain the two-model water waves.

Now let us pay attention to the following equation
\begin{eqnarray}\label{eq4.1}
& & F(\lambda,\beta,(\mu,w))  \\ \nonumber
& & :=\left( \lambda-\beta\left( \frac{[w^2]}{2h}-w+\mathcal{C}_{h}(ww')-w\mathcal{C}_{h}(w') \right) \right.\\ \nonumber
& &\left.-\frac{gA}{2}\left(\frac{[w^3]}{3h}-[w^2]-w^2-2hw+\mathcal{C}_{h}(w^2w')-w^2\mathcal{C}_{h}(w')+2h\mathcal{C}_{h}(ww')-2hw\mathcal{C}_{h}(w') \right) \right)^2 \\ \nonumber
& &-\left( \lambda^2+\mu-2gBw \right)\left(w'^{2}+\left(1+\mathcal{C}_{h}(w')\right)^2\right)-2\sigma\frac{w''+\mathcal{C}_{h}(w')w''
-w'\mathcal{C}_{h}(w'')}{((1+\mathcal{C}_{h}(w'))^2+w'^{2})^{\frac{1}{2}}}=0,
\end{eqnarray}
where we choose $\lambda,\beta$ as two parameters.  It is obvious that
$$F(\lambda,\beta,(0,0))=0$$
and the operator $F$ defined by (\ref{eq4.1}) is real-analytic from $ \mathbb{R}^3\times X\rightarrow Y$.
Based on the arguments in the previous section, we have that

\begin{lemma} \label{lem4.1}
For $(\lambda,\beta)\in \mathbb{R}^2$, the Frechet derivative $F_{(\mu,w)}(\lambda,\beta,(0,0)): \mathbb{R}\times X\rightarrow Y $ is a Fourier multiplier. Furthermore, there holds that
$$
F_{(\mu,w)}(\lambda,\beta,(0,0))\left(0, \sum_{n=1}^{\infty}a_{n}\cos(nx)\right)=\sum_{n=1}^{\infty} D(n,\lambda,\beta) a_{n}\cos(nx)
$$
with $D(n,\lambda,\beta)=-2\left( \frac{\lambda^2}{T_n}-\beta\lambda-gAh\lambda-(gB+\sigma n^2) \right)$ where $T_{n}=\frac{\tanh(nh)}{n}$.
Moreover, $D(n,\lambda_{n,\pm}^*,\beta)=0$ for $\lambda_{n,\pm}^*$ defined by (\ref{eq3.6}).
More precisely, for any $n,m\in \mathbb{N}\setminus \{0\}$  with $n\neq m$ and defining
\begin{equation}\label{eq4.2}
\beta_{n,m}:=\frac{\left( (gB+\sigma n^2)T_n-(gB+\sigma m^2)T_m \right)^2}{\sigma T_n T_m(T_n-T_m)(m^2-n^2)},
\end{equation}
we have that

(1) $\beta_{n,m}>0$;

(2) if $\lambda\notin \{\lambda_{n,\pm}^* : n\in \mathbb{N}\}$, then $F_{(\mu,w)}(\lambda,\beta,(0,0))$ is an invertible operator;

(3) if $\lambda=\lambda_{n,\pm}^*$ and and $\beta\neq\beta^*_{\pm}$ where
\begin{equation}\label{eq4.22}
\beta^*_{\pm}=-Agh\pm\sqrt{\beta_{m,n}},
\end{equation}
then $0$ is an eigenvalue of $F_{(\mu,w)}(\lambda,\beta,(0,0))$ and the corresponding eigenspace is one-dimensional;

(4) if $\lambda=\lambda_{n,\pm}^*$ and $\beta=\beta^*_{\pm}$,
then $\lambda_{n,+}^*=\lambda_{m,+}^*$ (or $\lambda_{n,-}^*=\lambda_{m,-}^*$). That is to say, $0$ is an eigenvalue of the operator $F_{(\mu,w)}(\lambda,\beta,(0,0))$ and the corresponding eigenspace is two-dimensional.
\end{lemma}
\begin{proof}
It follows from the sequence $(T_n)_n$ of being decreasing on $n$ that $\beta_{n,m}$ is a positive constant.
Note that zero is an eigenvalue of $F_{(\mu,w)}(\lambda, \beta,(0,0))$ if and only if $D(n,\lambda,\beta)=0$. Thus, if $\lambda\notin \{\lambda_{n,\pm}^* : n\in \mathbb{N}\}$, then $F_{(\mu,w)}(\lambda,\beta,(0,0))$ is an invertible operator.

On the other hand, if letting $\lambda_{n,+}^*=\lambda_{m,+}^*$ or $\lambda_{n,-}^*=\lambda_{m,-}^*$ for some $n\neq m$, we obtain by using algebraic manipulations that $(\beta+Agh)^2=\beta_{n,m}$. That is to say, if $(\beta+Agh)^2\notin\{\beta_{n,m}: m\neq n\}$, then $\lambda_{n,+}^*\neq\lambda_{m,+}^*$ and $\lambda_{n,-}^*\neq\lambda_{m,-}^*$ when $n\neq m$. Thus we can conclude that the kernel of $F_{(\mu,w)}(\lambda,\beta,(0,0))$ is generated by $\{\cos(nx)\}$ for any $\lambda\in\{\lambda_{n,+}^*,\lambda_{n,-}^*\}$. If $(\beta+Agh)^2\in\{\beta_{n,m}: m\neq n\}$, then $\lambda_{n,+}=\lambda_{m,+}$ (or $\lambda_{n,-}^*=\lambda_{m,-}^*$), which implies that the kernel space of $F_{(\mu,w)}(\lambda_{n,\pm}^*,\beta^*_{\pm},(0,0))$ is generated by $\{\cos(nx), \cos(mx)\}$.
\end{proof}

To deal with the new case $(4)$ in Lemma \ref{lem4.1}, we will use a multiparameter bifurcation Theorem \ref{thm6.4} with a high-dimensional kernel due to Kielh\"{o}fer. Thus, we can obtain the following existence result on two-mode water waves.

\begin{theorem} \label{thm4.2} (The existence of two-model water waves)
Let $\lambda^*\in\{\lambda^*_{n,\pm}\}$ defined by (\ref{eq3.6}) and $\beta^*\in\{\beta^*_\pm\}$ defined by (\ref{eq4.22}), then there exist a smooth solution curve
$$
\{(\lambda(s),\beta(s), (\mu(s), w(s))) |s\in (-\delta,\delta) \}
$$
to the problem (\ref{eq4.1}) with $(\lambda(0),\beta(0), (\mu(0), w(0)))=(\lambda^*,\beta^*, (0,0))$, where $\delta$ is small enough.
More precisely, we have that
$$
F(\lambda(s),\beta(s), (\mu(s), w(s)))=0,
$$
where $w(s)=s(a\cos(nx)+b\cos(mx))+O(s^2)$ for $a^2+b^2=1$.
\end{theorem}

\begin{proof}

To finish the proof, we will verify the condition $(1)$ and the condition $(2)$ in Theorem \ref{thm6.4}.
In our problem, the abstract Banach spaces will also be defined by $X:=C^{2,\alpha}_{2\pi,e}(\mathbb{R})$ and $Y:=C^{0,\alpha}_{2\pi,e}(\mathbb{R})$ and we claim that the condition $(1)$ holds naturally for $(\lambda,\beta)=(\lambda^*, \beta^*)$. Indeed, it follows easily from the analysis in Lemma \ref{lem4.1} that $\mathcal{N}(F_{(\mu,w)}(\lambda^*, \beta^*,(0,0)))=\{\cos{nx},\cos(mx)\}$. In addition, we have that the range $\mathcal{R}(F_{(\mu,w)}(\lambda^*,\beta^*,(0,0)))$ is the closed subset of $Y$ consisting of the elements $\varphi$ with
$$
\int_{-\pi}^{\pi}\varphi\cos(nx)dx=0\quad \text{and} \quad  \int_{-\pi}^{\pi}\varphi\cos(mx)dx=0.
$$
Thus, the co-range $Y\setminus \mathcal{R}(F_{(\mu,w)}(\lambda^*, \beta^*,(0,0)))$ is generated by $\{\cos(nx),\cos(mx)\}$ which is also two-dimensional.

To check the validity of the condition $(2)$ in Theorem \ref{thm6.4}, we now take
$$\hat{v}=a\cos(nx)+b\cos(mx)$$
with nontrivial constants $a,b$ and $a^2+b^2=1$.
Moreover, it is known that a complement of $\mathcal{R}(F_{(\mu,w)}(\lambda^*,\beta^*,(0,0)))$ is
$$
Y_0=:\text{span}\{\cos(nx),\cos(mx)\}.
$$
From Lemma \ref{lem4.1}, we compute
\begin{eqnarray}\label{eq4.23}
& & F_{\lambda(\mu,w)}(\lambda^*, \beta^*,(0,0)))(a\cos(nx)+b\cos(mx)) \\ \nonumber
& & =a\left(2\beta^*+2gAh-\frac{4\lambda^*}{T_n}\right)\cos(nx)+ b\left(2\beta^*+2gAh-\frac{4\lambda^*}{T_m}\right)\cos(mx)
\end{eqnarray}
and
\begin{eqnarray}\label{eq4.24}
& & F_{\beta(\mu,w)}(\lambda^*, \beta^*,(0,0)))(a\cos(nx)+b\cos(mx)) \\ \nonumber
& & =2\lambda^*a\cos(nx)+2\lambda^*b\cos(mx)
\end{eqnarray}
It is not difficult to check that the coefficients
$$
a\left(2\beta^*+2gAh-\frac{4\lambda^*}{T_n}\right), \quad b\left(2\beta^*+2gAh-\frac{4\lambda^*}{T_m}\right)
$$
and
$$
2\lambda^*a, \quad 2\lambda^*b
$$
in (\ref{eq4.23}) and (\ref{eq4.24}) are nonzero due to the definition $\{\lambda^*_{n,\pm}\}$ in (\ref{eq3.6}) and $\{\beta^*_\pm\}$ in (\ref{eq4.22}) and the nontrivial $a,b$.
In addition, it follows from the sequence $(T_n)_n$ of being decreasing on $n$ that the vectors
$$
a\left(2\beta^*+2gAh-\frac{4\lambda^*}{T_n}\right)\cos(nx)+ b\left(2\beta^*+2gAh-\frac{4\lambda^*}{T_m}\right)\cos(mx)
$$
and
$$
2\lambda^*a\cos(nx)+2\lambda^*b\cos(mx)
$$
are linearly independent, which implies that the condition $(2)$ is also satisfied.
Then we can obtain the existence result of two-model water waves by using the Theorem \ref{thm6.4}.
\end{proof}

\begin{remark}
It is worth noting that, following the ideas in \cite{MartinM}, we can obtain the ripples solutions for stratified flows by using the Secondary bifurcation theorem due to \cite{Shearer} with checking additionally the following nondegeneracy conditions
\begin{equation}  \nonumber\\
\left|
\begin{array}{cc}
\langle F_{\lambda (\mu,w)}(\lambda^*,\beta^*,(0,0))[1,w_{1}]|\psi'_{1} \rangle  ~&~  \langle F_{\beta (\mu,w)}(\lambda^*,\beta^*,(0,0))[1,w_{1}]|\psi'_{1} \rangle \\
~ & ~\\
\langle F_{\lambda (\mu,w)}(\lambda^*,\beta^*,(0,0))[1,w_{2}]|\psi'_{2} \rangle   ~&~ \langle F_{\beta (\mu,w)}(\lambda^*,\beta^*,(0,0))[1,w_{2}]|\psi'_{2} \rangle  \\
\end{array}
\right|\neq0
\end{equation}
and
\begin{equation}  \nonumber\\
\left|
\begin{array}{cc}
\langle F_{\lambda (\mu,w)}(\lambda^*,\beta^*,(0,0))[1,w_{1}]|\psi'_{1} \rangle   ~&~  \langle F_{(\mu,w)(\mu,w)}(\lambda^*,\beta^*,(0,0))[w_{1},w_{1}]|\psi'_{1} \rangle \\
~ & ~\\
\langle F_{\lambda (\mu,w)}(\lambda^*,\beta^*,(0,0))[1,w_{2}]|\psi'_{2} \rangle   ~&~ \langle F_{(\mu,w)(\mu,w)}(\lambda^*,\beta^*,(0,0))[w_{1},w_{2}]|\psi'_{2} \rangle  \\
\end{array}
\right|\neq0,
\end{equation}
where $\langle\psi_i |\psi'_{j}\rangle=\delta_{ij} $ and $\mathcal{N}(\psi'_{i})=\mathcal{R}(F_{w}(\lambda^*,\beta^*,0))$ for $i=1,2$ with $\langle\cdot|\cdot\rangle$ denoting the duality pairing on $Y\times Y'$. However, the calculation process is extremely tedious and complicated. We also would like to mention that the two-model water waves obtained in Theorem \ref{thm4.2} are slightly different from the ripples. More precisely, a
smooth secondary bifurcation branch, consisting of Wilton ripples, emerges from each of the primary branches when the Bernoulli constant $\beta$ is
sufficiently close to the critical values $\beta^*_{\pm}$ mentioned before. While the local bifurcation curves, consisting of two-model water waves, are the primary branches with the combination $a\cos(nx)+b\cos(mx)$ as a new kernel.
In addition, as far as we know that it is still open to obtain the local stability of the ripples or the two-model water waves. Maybe some new exchange of
stability theorems similar as Theorem \ref{thm6.3} are needed to develop, which is a very interesting project
in future.
\end{remark}

\section{The waves with stagnation points}
In this section, we will discuss the occurrence of stagnation points and critical layers inside the stratified linearly water waves established in previous Section. Based on the argument above, for each choice of the parameters $(\lambda,\beta)\in \mathbb{R}^2$, the pair $(\mu,w)=(0,0)$ is a trivial solution to the stratified steady periodic water waves problem (\ref{eq4.1}). These are the laminar flow solutions, which are corresponding to the waves with a flat surface and parallel streamlines. The constant $\lambda$ can be regarded as the speed at the surface and it follows from (\ref{eq3.2}) that the stream function $\psi(X,Y)$ associated with these trivial solutions is given by
\begin{equation} \label{eq5.1}
\psi(X,Y)=\frac{\beta}{2}Y^2+\frac{gA}{6}Y^3+\left( \frac{m}{h}-\frac{\beta h}{2}-\frac{gAh^2}{6} \right)Y-m, ~~for~~0\leq Y\leq h
\end{equation}
with the pseudo velocity field
\begin{equation}\label{eq5.2}
(\psi_{Y},-\psi_{X})=\left(\beta Y+\frac{Ag}{2}Y^{2}+\frac{m}{h}-\frac{\beta h}{2}-\frac{Agh^{2}}{6},~0\right),~~for~~0\leq Y\leq h.
\end{equation}
Furthermore, we can write (\ref{eq5.2}) as
\begin{equation}\label{eq5.3}
(\psi_{Y},-\psi_{X})=\left((Y-h)\left( \frac{Ag}{2}(Y+h)+\beta  \right)+\lambda,~0\right),~~for~~0\leq Y\leq h,
\end{equation}
For the laminar flows, we can deduce from (\ref{eq5.3}) that
\begin{itemize}
  \item If $Y=h$ or $\beta=-\frac{Ag}{2}(Y+h)$, the existence of stagnation points is impossible near the bifurcation points $\lambda_{n,\pm}^{*}$ obtained in (\ref{eq3.6}) due to $\lambda_{n,+}^{*}>0$ and $\lambda_{n,-}^{*}<0$. In fact, this implies that there are not stagnation points on the surface $Y=h$ near $\lambda_{n,\pm}^{*}$.
  \item If $\beta>-\frac{Ag}{2}(Y+h)$, it follows that stagnation points can occur inside the fluid $Y<h$ near $\lambda_{n,+}^*$.
  \item If $\beta<-\frac{Ag}{2}(Y+h)$, it follows that stagnation points can occur inside the fluid $Y<h$ near $\lambda_{n,-}^*$
\end{itemize}
By perturbation, it's obvious that if the wave amplitude is small there exist stagnation points on the streamline inside the stratified periodic water waves away from the free surface. On the other hand, it follows from Section 4 that the two-model water waves can occur when
$$\beta\rightarrow \beta^*_\pm=-Ahg\pm \sqrt{\frac{\left( (gB+\sigma n^2)T_n-(gB+\sigma m^2)T_m \right)^2}{\sigma T_n T_m(T_n-T_m)(m^2-n^2)}}.  $$
It is obvious that
$$
\beta^*_\pm=-Ahg- \sqrt{\frac{\left( (gB+\sigma n^2)T_n-(gB+\sigma m^2)T_m \right)^2}{\sigma T_n T_m(T_n-T_m)(m^2-n^2)}}<-Ahg\leq-\frac{Ag}{2}(Y+h)
$$
for $0\leq Y\leq h$. This implies the existence of stagnation points inside two-model waves as $\lambda$ is close to $\lambda_{n,-}^*$.
Then it follows from similar argument \cite[Section 5]{ConstantinSV}, the critical layers and cat's eye structure would occur in two-model water waves as follows.
\begin{figure}[ht]
\centering
\includegraphics[width=0.75\textwidth]{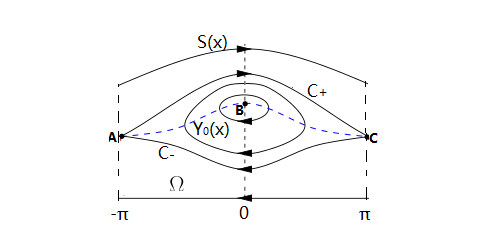}
\caption{The "Cat's eye" flow pattern and the critical layer depicted by the dashed curve, where stagnation points "\textbf{A, B, C}" on the critical layer occurs at the crest (or trough) line.}
\end{figure}

\section*{Acknowledgments}
Wang was supported by National Key R$\&$D Program of China
(No. 2022YFA1005601), National Natural Science Foundation of China (No. 12371114) and Outstanding Young
foundation of Jiangsu Province (No. BK20200042). Xu was supported by the Postdoctoral Science Foundation of China (No. 2023M731381). Zhang was supported by National Natural Science Foundation of China (No. 12301133), the Postdoctoral Science Foundation of China (No. 2023M741441) and Jiangsu Education Department (No. 23KJB110007).

\section*{Data Availability Statements}
Data sharing not applicable to this article as no datasets were
generated or analysed during the current study.

\section*{Competing interests}
The authors declare that they have no
competing interests.

\section{Appendix: Quoted results}
This appendix collects the remarkable Crandall and Rabinowitz local bifurcation theorem \cite{CrandallR}, which are designed to deal with the case of linearized operator with one dimensional kernel and recalls the stability exchange theorem due to Crandall and Rabinowitz \cite{CrandallR1}. At last, we also include a multiparameter bifurcation theorem with a high-dimensional kernel due to Kielh\"{o}fer \cite{Kielhofer}.

\begin{theorem}(Local bifurcation theorem) \label{thm6.1}
Let $X$ and $Y$ be Banach spaces and $F: \mathbb{R}\times X \rightarrow Y$ be a $C^k$ function with $k\geq 2$. Suppose that

(H1) $F(\lambda,0)=0$ for all $\lambda\in \mathbb{R}$;

(H2) For some $\lambda_*\in \mathbb{R}$, $ F_{x}\left(\lambda_*, 0\right)$ is a Fredholm operator with $\mathcal{N} \left( F_{x}\left(\lambda_*, 0\right)\right)$ and $Y /\mathcal{R}\left( F_{x}\left(\lambda_*, 0\right)\right)$ are $1$-dimensional and
the null space generated by $x_*$, and the transversality condition
\begin{equation}
F_{\lambda,x}\left(\lambda_*, 0\right)(1,x_*)\not\in\mathcal{R}\left( F_{x}\left(\lambda_*, 0\right)\right)\nonumber
\end{equation}
holds, where $\mathcal{N} \left( F_{x}\left(\lambda_*, 0\right)\right)$ and $\mathcal{R}\left( F_{x}\left(\lambda_*, 0\right)\right)$ denote null space and range space of $ F_{x}\left(\lambda_*, 0\right)$, respectively.

Then $\lambda_*$ is a bifurcation point in the sense that there exists $\varepsilon>0$ and a primary branch $\mathcal{P}$ of solutions
$$
\{ (\lambda,x)=(\Lambda(s),s\psi(s)):s\in\mathbb{R}, |s|<\varepsilon \}\subset\mathbb{R}\times X
$$
with $F(\lambda,x)=0$, $\Lambda(0)=0$, $\psi(0)=x_*$ and the maps
$$
s\mapsto\Lambda(s)\in\mathbb{R},~~~~s\mapsto s\psi(s)\in X
$$
are of class $C^{k-1}$ on $(-\varepsilon, \varepsilon)$.
\end{theorem}

\begin{theorem}(The formula of bifurcation direction \cite{Kielhofer}) \label{thm6.2}
\emph{If $F$ satisfies the hypotheses of Theorem \ref{thm6.1}, then
\begin{equation}
\lambda'(0)=-\frac{\left\langle l,F_x^{(2)}\left(\lambda_*,0\right)x_*^2\right\rangle}{2\left\langle l, F_{\lambda x}(\lambda_*,0)x_*\right\rangle},\nonumber
\end{equation}
where $l\in X^*$ satisfying} $\mathcal{N}(l)=\mathcal{R}\left(D_x F(\lambda_*,0)\right)$, \emph{and $X^*$ being the dual space of $X$.
Furthermore, for $n\geq2$, if $F_x^{(j)}(\lambda_*,0)x_*^j=0$ for $j\in\{1,\ldots,n\}$, then $\lambda^{(j)}(0)=0$, $\psi^{(j)}(0)=0$ for $j\in\{1,\ldots,n-1\}$} and
\begin{equation}
\lambda^{(n)}(0)=-\frac{\left\langle l,F_x^{(n+1)}\left(\lambda_*,0\right)x_*^{n+1}\right\rangle}{(n+1)\left\langle l, F_{\lambda x}(\lambda_*,0)x_*\right\rangle},\nonumber
\end{equation}
\emph{where $F_x^{(j)}(\lambda_*,0)x_*^j$ means the value of the $j$-th Fr\'{e}chet derivative of $F(\lambda_*,x)$ with respect to $x$ at $(\lambda_*,0)$}.
\end{theorem}

We recall the concept of formal stability and the definition of $K-$simple eigenvalue introduced in \cite{CrandallR1}. The operator equation $F(\lambda,x)=0$ can be regarded as the equilibrium form of the evolution equation
\begin{eqnarray}\label{eq6.1}
\frac{dx}{dt}=F(\lambda,x).
\end{eqnarray}
Assume that $F\left(\lambda_0, x_0\right)=0$. If all of eigenvalues of $F_x\left(\lambda_0, x_0\right)$ are negative, $x_0$ is called asymptotically formal stable solution of (\ref{eq6.1}). If there exists a positive eigenvalue of $F_x\left(\lambda_0, x_0\right)$, $x_0$ is called unstable. If there exists a zero eigenvalue of $F_x\left(\lambda_0, x_0\right)$, $x_0$ is called neutral stable.
Let $\mathcal{N}(T)$ and $\mathcal{R}(T)$ denote the null space and range of any operator $T$.

\begin{definition}

\emph{Let $T,K$: $X\rightarrow Y$ be two bounded linear operators from a real Banach space $X$ to another one $Y$. A complex number $\theta$ is called a $K-$simple eigenvalue of $T$ if
$$
dim \mathcal{N}(T-\theta K)=1=codim \mathcal{R}(T-\theta K)
$$
and
$$
K \psi^* \notin \mathcal{R}(T-\theta K) ~~for ~~0\neq \psi^*\in \mathcal{N}(T-\theta K).
$$}
\end{definition}

If $K$ is the identity operator, $K-$simple eigenvalue is called simple.
Now let us recall the Crandall-Rabinowitz exchange of stability theorem \cite{CrandallR1}.

\begin{theorem} \label{thm6.3}

\emph{Let $X$ and $Y$ be real Banach spaces and let $K, T: X\rightarrow Y$ be two bounded linear operators. Assume $F: \mathbb{R}\times X\rightarrow Y$ is $C^{2}$ near $\left(\lambda_*,0\right)\in \mathbb{R}\times X$ with $F(\lambda,0)=0$ for $\left\vert\lambda_*-\lambda\right\vert$ sufficiently small. Let $T=F_x\left(\lambda_*,0\right)$. If $\theta=0$ is a $F_{\lambda x}\left(\lambda_*,0\right)-$simple eigenvalue of operator $T$ and a $K-$simple eigenvalue of $T$, then there exists locally a curve $(\lambda(s),x(s))\in \mathbb{R}\times X$ such that
$$
(\lambda(0),x(0))=\left(\lambda_*,0\right)~~and ~~F(\lambda(s),x(s))=0.
$$
Moreover, if $F(\lambda,x)=0$ with $x\neq 0$ and $(\lambda,x)$ near $\left(\lambda_*,0\right)$, then
$$
(\lambda,x)=(\lambda(s),x(s))~~for~~some~~s\neq 0.
$$
Furthermore, there are eigenvalues $\theta(s)$, $\theta_{triv}(\lambda)\in \mathbb{R}$ with eigenvectors $\psi(s)$, $\psi_{triv}(\lambda)\in X $, such that
\begin{eqnarray}
F_{x}\left(\lambda(s),x(s)\right)\psi(s)=\theta(s)K\psi(s),\nonumber
\end{eqnarray}
\begin{eqnarray}
F_{x}\left(\lambda,0\right)\psi_{triv}(\lambda)=\theta_{triv}(\lambda)K\psi_{triv}(\lambda)\nonumber
\end{eqnarray}
with
\begin{align*}
\theta(0)=\theta_{triv}\left(\lambda_*\right)=0,~~\psi(0)=\psi_{triv}\left(\lambda_*\right)=\psi^*.
\end{align*}
Each curve is $C^{1}$ if $F$ is $C^{2}$, then
\begin{eqnarray}
\frac{d\theta_{triv}\left(\lambda\right)}{d\lambda}|_{\lambda=\lambda_*}\neq 0,~~\lim_{s\rightarrow 0,\theta\left(s\right)\neq 0}\frac{s\lambda'(s)}{\theta(s)}=-\frac{1}{\theta_{triv}'\left(\lambda_*\right)}.\nonumber
\end{eqnarray}
}
\end{theorem}

Since $\frac{d\theta_{triv}\left(\lambda\right)}{d\lambda}|_{\lambda=\lambda_*}\neq 0$ and $\theta_{triv}\left(\lambda_*\right)=0$, the trivial solution is stable formally at one side of $\lambda_*$ and is unstable at another side of $\lambda_*$.
Based on the arguments above, the nontrivial solution $x(s)$ is called to be stable formally if $\theta(s)<0$ and unstable if $\theta(s)>0$.

\begin{theorem} \label{thm6.4}
Let $X$ and $Y$ be real Banach spaces, $U$ be a neighborhood of $0$ in $X$ and assume that $F\in C^2(\Lambda\times U, V)$ with $F(\lambda,0)=0$ for any $\lambda\in \Lambda$, where $\lambda=(\lambda_1,\lambda_2,...,\lambda_k)$ and $\Lambda$ is an open set in $\mathbb{R}^k$ for $k\in \mathbb{N}^+$. Suppose that

(1) $\text{dim} \mathcal{N}\left(F_x(\mu,0)\right)=\text{codim} \mathcal{R}\left( F_x(\mu,0)\right)=k$ for some $\mu=(\mu_1,\mu_2,...,\mu_k)\in \Lambda$ and $\left\{x_1,x_2,...,x_k\right\}$ are the basis of $\mathcal{N}\left( F_x(\mu,0)\right)$;

(2) there is a $\hat{v}\in\mathcal{N}\left(F_x(\mu,0)\right)$ with $\|\hat{v}\|=1$ such that a complement $Y_0$ of $\mathcal{R}\left( F_x(\mu,0)\right)$ is spanned by the vectors $F_{\lambda_ix}(\mu,0)\hat{v}$ for $i=1,...,k$;\\
Then the equation $F(\lambda, x)=0$ possesses a continuously differentiable solution curve
$$
\{(\lambda(s), x(s) ) |s \in (-\delta, \delta)\}\subset  \mathbb{R}^k\times X
$$
through $(\lambda(0),x(0))=(\mu, 0)$ and $\frac{dx}{ds}(0)=\hat{v}$.

\end{theorem}

\end{document}